\documentclass[12pt,a4paper]{article}
\usepackage{theorem}
\newtheorem{lemma}{Lemma}

\newtheorem{theorem}[lemma]{Theorem}
\newtheorem{corollary}[lemma]{Corollary}
\newtheorem{example}[lemma]{Example}

\newcommand{\Z}{{\bf Z}}
\newcommand{\R}{{\bf R}}
\newcommand{\C}{{\bf C}}

\newcommand{\rme}{{\rm e}}
\newcommand{\rmd}{{\rm d}}
\newcommand{\cE}{{\cal E}}

\newcommand{\cG}{{\cal G}}

\newcommand{\cL}{{\cal L}}
\newcommand{\cM}{{\cal M}}

\newcommand{\cS}{{\cal S}}
\newcommand{\cT}{{\cal T}}

\newcommand{\sig}{\sigma}

\newcommand{\bet}{\beta}
\newcommand{\gam}{\gamma}

\newcommand{\lam}{\lambda}
\newcommand{\del}{\delta}
\newcommand{\eps}{\varepsilon}

\newcommand{\tr}{{\rm tr}}

\newcommand{\Spec}{{\rm Spec}}

\newcommand{\norm}{\Vert}
\renewcommand{\Re}{{\rm Re}}
\renewcommand{\Im}{{\rm Im}}

\newenvironment{eq}{\begin{equation}}{\end{equation}}
\newcommand{\block}{\hfill\rule{2.5mm}{2.5mm}}
\newenvironment{proof}{\textbf{Proof}}{\block}
\usepackage{hyperref}
\usepackage{graphicx}
\title{Embeddable Markov Matrices}
\author{E B Davies}
\date{11 January 2010}
\begin{document}
\maketitle
%%%%%%%%%%%%%%%%%%%%%
\begin{abstract}
We give an account of some results, both old and new, about any $n\times n$ Markov matrix that is embeddable in a one-parameter Markov semigroup. These include the fact that its eigenvalues must lie in a certain region in the unit ball. We prove that a well-known procedure for approximating a non-embeddable Markov matrix by an embeddable one is optimal in a certain sense.
\end{abstract}

SHORT TITLE: Embeddable Markov Matrices

\section{Introduction}\label{intro}

A Markov matrix $A$ is defined to be a real $n\times n$ matrix with
non-negative entries satisfying $\sum_{j=1}^n A_{i,j}=1$ for all
$i$. The spectral properties of non-negative matrices and linear
operators and in particular of Markov matrices have been studied in
great detail, because of their great importance in finance,
population dynamics, medical statistics, sociology and many other
areas of probability and statistics. Theoretical accounts of parts
of the subject may be found in  \cite{chung,LOTS,minc,nagel}. This
paper develops ideas of \cite{higham}, which investigated when the
$p$th roots of Markov matrices were also Markov; this problem is
related to the possibility of passing from statistics gathered at
certain time intervals, for example every year, to the corresponding
data for shorter time intervals.

Given an empirical Markov matrix, three major issues discussed in
\cite{singer} are embeddability, uniqueness of the embedding and the
effects of data/sampling error. All of these are also considered
here. We call a Markov matrix $A$ embeddable if there exists a
matrix $B$ such that $A=\rme^B$ and $\rme^{Bt}$ is Markov for all
$t\geq 0$. The matrix $B$ involved need not be unique, but it must
have non-negative off-diagonal entries and all its row sums must
vanish; see \cite{chung} or \cite[section 12.3]{LOTS}. In
probabilistic terms a Markov matrix  $A$ is embeddable if it is
obtained by taking a snapshot at a particular time of an autonomous
finite state Markov process that develops continuously in time. On
the other hand a Markov matrix might not be embeddable if it
describes the annual changes in a population that has a strongly
seasonal breeding pattern; in such cases one might construct a more
elaborate model that incorporates the seasonal variations.
Embeddability may also fail because the matrix entries are not
accurate; in such cases a regularization technique might yield a
very similar Markov matrix that is embeddable; see \cite{KS} for
examples arising in finance.

Theorem~\ref{maintheorem} describes some spectral consequences of embeddability. The earliest analysis of the structure of the set $\cE$ of embeddable $n\times n$ Markov matrices and its topological boundary in the set of all Markov matrices was given by Kingman \cite{kingman}, who concluded that except in the case $n=2$ it seemed unlikely that any very explicit characterisation of $\cE$ could be given; see \cite{johan} for further work on this problem. Theorem~\ref{approxtheorem} proves that a well-known method of approximating a Markov matrix by an embeddable Markov matrix is optimal in a certain sense. Many of the results in the present paper appear in one form or other in papers devoted to the wide variety of applications, and it is hoped that collecting them in one place may be of value.

\section{The main theorem}

For the sake of definiteness we define the principal logarithm of a
number $z\in\C\backslash (-\infty,0]$ to be the branch of the
logarithm with values in $\{ w:|\Im(w)|<\pi\}$. We define the
principal logarithm of an $n\times n$ matrix $A$ such that
$\Spec(A)\cap (-\infty,0]=\emptyset$ to be that defined by the
functional integral
\begin{eq}
\log(A)=\frac{1}{2\pi i}\int_\gam \frac{\log(z)}{zI-A}\, \rmd
z\label{logcontour}
\end{eq}%
using the principal logarithm of $z$ and a simple closed contour
$\gam$ in $\C\backslash (-\infty,0]$ that encloses the spectrum of
$A$. This formula goes back to Giorgi in 1926; see
\cite[Theorem~VII.1.10 and notes, p.607]{DS}. If $A=TDT^{-1}$ where
$D$ is diagonal, this is equivalent to $\log(A)=T\log(D)T^{-1}$
where $\log(D)$ is obtained from $D$ by applying $\log$ to each
diagonal entry of $D$. The non-diagonalisable case is discussed in
some detail in \cite{singer} and yields the same matrix as
(\ref{logcontour}).

\begin{lemma}\label{Markovlog}
If $A$ is a Markov matrix and $\Spec(A)\cap (-\infty,0]=\emptyset$
then the principal logarithm $L=\log(A)$ lies in the set $\cL$ of
all real $n\times n$ matrices $L$ such that $\sum_{1\leq j\leq n}
L_{i,j}=0$ for every $i$.
\end{lemma}

\begin{proof}
We use the formula (\ref{logcontour}) and take the contour $\gam$ to
be symmetrical with respect to reflection about the $x$-axis. The
statements of the lemma follow directly from two properties of the
resolvent matrices.

The first is the identity
\begin{eq}%
((\overline{z}I-A)^{-1})_{i,j}=\overline{(zI-A)^{-1})_{i,j}}\label{resolventconj}
\end{eq}%
This holds for large $|z|$ by virtue of the identity
\begin{eq}%
(zI-A)^{-1}=z^{-1}\sum_{n=0}^\infty (A/z)^n\label{resolventseries}
\end{eq}%
and (\ref{resolventconj}) then extends to all $z\notin \Spec(A)$ by
analytic continuation.

The second identity needed is
\[
(zI-A)^{-1}1=(z-1)^{-1}1,
\]
whose proof follows the same route, using (\ref{resolventseries})
and analytic continuation.
\end{proof}

The results in our next lemma are all well known and are included
for completeness.

\begin{lemma}
If $A$ is embeddable then $0$ is not an eigenvalue of $A$ and every
negative eigenvalue has even algebraic multiplicity. Moreover
$\det(A)>0$. If $A$ is embeddable and $A_{i,j}>0$, $A_{j,k}>0$ then
$A_{i,k}>0$.
\end{lemma}

\begin{proof}
The first statement follows from the fact that
\[
\Spec(A)=\exp(\Spec(B)).
\]

Given an eigenvalue $\lam<0$ of $A$ let
\begin{eqnarray*}
S_+&=&\{ z\in\Spec(B): \rme^z=\lam \mbox{ and } \Im(z)>0\},\\
S_-&=&\{ z\in\Spec(B): \rme^z=\lam \mbox{ and } \Im(z)<0\},
\end{eqnarray*}
and let $\cL_\pm$ be the spectral projections of $B$ associated with
$S_\pm$. Since $\rme^z=\lam$ implies that $\Im(z)\not= 0$, we can
deduce that $\cL_-\cap\cL_+ =0$ and that $\cM=\cL_-+\cL_+$ is the
spectral projection of $A$ associated with the eigenvalue $\lam$.
Since $B$ is real $\cL_-$ may be obtained from $\cL_+$ by complex
conjugation, so
\[
\dim{\cM}=\dim(\cL_+)+\dim(\cL_-)=2\dim(\cL_+).
\]
See \cite{elfving}.

By combining the reality of $B$ with the formula
$\,\det(A)=\rme^{\tr(B)}$ we obtain $\det(A)>0$. See \cite{kingman}.

The last statement follows from the general theory of Markov chains
and is due to Ornstein and Levy, independently; see \cite[Section
2.5, Theorem 2]{chung} and \cite[Theorem~13.2.4]{LOTS}. We first
note that one may write $B=C-\del I$ where all the entries of $C$
are non-negative and $\del \geq 0$. Hence
\[
\rme^{Bt}=\rme^{-\del t}\sum_{n=0}^\infty C^n t^n/n!
\]
where each entry of each $C^n$ is non-negative. This implies that if
$(\rme^{Bt})_{i,j}>0$ for some $t>0$ the same holds for all $t>0$.
This quickly yields the final statement.
\end{proof}

Kingman \cite{kingman} has shown that the set $\cE$ of embeddable
Markov matrices is a closed subset of the set of all $n\times n$
Markov matrices. The matrix norm used throughout this paper is
\begin{eq}%
\norm M\norm=\max\{ \norm Mv\norm_\infty:\norm v\norm_\infty\leq 1\} =\max_{1\leq i\leq n}\left\{ \sum%
\!{\rule[-0.6em]{0em}{1.8em}}%
_{j=1}^n |M_{i,j}|\right\}.\label{matrixnorm}
\end{eq}

\begin{lemma}\label{density}
The set $\cS$ of all $A\in\cE$ with no negative eigenvalues is a
dense relatively open subset of $\cE$.
\end{lemma}

\begin{proof}
If $A\in\cS$ then a simple perturbation theoretic argument implies
that there exists $\eps>0$ such that $C$ has no negative eigenvalues
for any $n\times n$ matrix satisfying $\norm A-C\norm <\eps$. This
implies that $\cS$ is relatively open in $\cE$.

If $A\in\cE$ then $A=\rme^B$ for some Markov generator $B$. If
$\{x_r+iy_r\}_{r=1}^n$ is the set of eigenvalues of $B$ and $t> 0$
then $\rme^{Bt}$ has a negative eigenvalue if and only if
$ty_r=\pi(2m+1)$ for some $r$ and some integer $m$. The set of such
$t$ is clearly discrete. It follows that $\rme^{Bt} \in \cS$ for all
$t$ close enough to $1$ except possibly for $1$ itself. Since
$\lim_{t\to 1}\norm \rme^{Bt}-A\norm =0$, we conclude that $\cS$ is
dense in $\cE$.

\end{proof}

The following example shows that the density property in
Lemma~\ref{density} depends on the embeddability hypothesis.

\begin{example} \end{example}
The Markov matrix
\[
A=\left[ \begin{array}{cc}
1/3&2/3\\
2/3&1/3
\end{array}\right]
\]
has $\Spec(A)=\{ 1,-1/3\}$. If $0<\eps <1/3$, any matrix close
enough to $A$ also has a single eigenvalue $\lam$ satisfying
$|\lam+1/3|<\eps$ by a standard perturbation theoretic argument.
Since $A$ has real entries the complex conjugate of $\lam$ is also
an eigenvalue, so $\lam$ must be real and negative. Therefore the
set of Markov matrices with no negative eigenvalues is relatively
open but not dense in the set of all Markov matrices, at least for
$n=2$. The example may be used to construct a similar example for
every $n> 2$.\block

We will need Lemma~\ref{matrixdiscr} and its corollary in the proof
of Theorem~\ref{embeddistinct}.

\begin{lemma}\label{matrixdiscr}
There exists a polynomial $p$ in the coefficients of an $n\times n$
matrix $A$ such that $A$ has a multiple eigenvalue in the algebraic
sense if and only if $p=0$. We call $p$ the discriminant of $A$.
\end{lemma}

\begin{proof}
$A$ has a multiple eigenvalue if and only if its characteristic
polynomial $q(z)=z^n+a_1z^{n-1}+\ldots +a_n$ has a multiple root;
the coefficients of $q$ are themselves polynomials in the entries of
$A$. Moreover $q$ has a multiple root if and only if its
discriminant (the square of its Vandermonde determinant) vanishes,
and the discriminant of $q$ is a polynomial in $a_1,\ldots, a_n$.
\end{proof}

\begin{corollary}\label{cordiscr}
If $A_0,\, A_1$ are two $n\times n$ matrices then either
$A_z=(1-z)A_0+zA_1$ has a multiple eigenvalue for all $z\in\C$ or
this happens only for a finite number of $z$.
\end{corollary}

\begin{proof}
The discriminant of $A_z$ is a polynomial in $z$, which has a finite
number of roots unless it vanishes identically.
\end{proof}

\begin{theorem}\label{embeddistinct}
The set $\cT$ of all $n\times n$ embeddable Markov matrices that
have $n$ distinct eigenvalues is relatively open and dense in the
set $\cE$ of all embeddable Markov matrices.
\end{theorem}

\begin{proof}
A standard argument from perturbation theory establishes that $\cT$
is relatively open in $\cE$, so we only need to prove its density.

Let $A=\rme^{B_0}$ where $B_0$ is a Markov generator, and let
$\eps>0$. Then put $B_t=(1-t)B_0+tB_1$ where
\[
(B_1)_{r,s}=\left\{ \begin{array}{cc}
-1&\mbox{ if } r=s,\\
1&\mbox{ if } r+1=s,\\
1&\mbox{ if } r=n,\, s=1,\\
0&\mbox{ otherwise. }
\end{array}\right.
\]
One sees immediately that $B_t$ is a Markov generator for all $t\in
[0,1]$ and that it has $n$ distinct eigenvalues if $t=1$.
Corollary~\ref{cordiscr} now implies that the eigenvalues of $B_t$
are distinct for all sufficiently small $t>0$. By further
restricting the size of $t>0$ we may also ensure that $\norm
\rme^{B_t}-A\norm <\eps/2$.

Having chosen $t$, we put $B=sB_t$ where $s\in\R$ is close enough to
$1$ so that $\norm \rme^{B_t}-\rme^B\norm <\eps/2$; we also choose
$s$ so that if $\lam_1,\, \lam_2$ are any two eigenvalues of $B_t$
then $s(\lam_1-\lam_2)\notin 2\pi i\Z$. These conditions ensure that
$\norm \rme^B-A\norm<\eps$ and that $\rme^B$ has $n$ distinct
eigenvalues.
\end{proof}

The following lemma may be contrasted with the fact that a complex
number $\lam$ such that $|\lam |=1$ is the eigenvalue of some
$n\times n$ Markov matrix if and only if $\lam^r=1$ for some $r\in
\{1,2,\ldots,n\}$; see \cite[Chap~7, Theorem~1.4]{minc}. Permutation
matrices provide examples of such spectral behaviour. The lemma has
been extended to an infinite-dimensional context in \cite{EBDtriv}.

\begin{lemma}[Elfving, \cite{elfving}] \label{periph}
If $A$ is an embeddable Markov matrix and $\lam\not= 1$ is an
eigenvalue of $A$ then $|\lam|<1$.
\end{lemma}

\begin{proof}
Our hypotheses imply by \cite[Lemma~12.3.5]{LOTS} that
$\Spec(A)=\exp(\Spec(B))$ where $B=c(C-I)$, $c>0$ and $C$ is a
Markov matrix. Since $C$ is a contraction when considered as acting
in $\C^n$ with the $l^\infty$ norm, $\Spec(C)\subset \{ z:|z|\leq
1\}$. Therefore every eigenvalue $\lam$ of $B$ except $0$ satisfies
$\Re(\lam)<0$. The lemma follows.
\end{proof}

The main application of the following theorem may be to establish
that certain Markov matrices arising in applications are not
embeddable, and hence either that the entries are not numerically
accurate or that the underlying process is not autonomous. The
theorem is a quantitative strengthening of Lemma~\ref{periph}. It is
of limited value except when $n$ is fairly small, but this is often
the case in applications.

\begin{theorem}[Runnenberg, \cite{runn, singer}]\label{maintheorem}
If $n\geq 3$ and the $n\times n$ Markov matrix $A$ is embeddable
then its spectrum is contained in the set
\[
\{ r\rme^{i\theta}:-\pi\leq \theta\leq \pi, 0<r\leq r(\theta)\}
\]
where
\[
r(\theta)= \exp(-\theta \tan(\pi/n)).
\]
\end{theorem}

\begin{proof}
This depends on two facts, firstly that $\Spec(A)=\exp(\Spec(B))$
where $B=c(C-I)$, $c>0$ and $C$ is a Markov matrix. Secondly
\begin{eqnarray}
\Spec(C-I) &\subseteq & \{ z:|\arg(z)|\geq \pi/2+\pi/n\}\label{argbound}\\
&=& \{ -u+iv: u\geq 0, |v|\leq u\cot(\pi/n)\}.\label{karpbound}
\end{eqnarray}
by applying a theorem of Karpelevi\v{c} to $C$ and then deducing

(\ref{karpbound}) from that; see \cite{karp} or \cite[Chap.~7,
Theorem~1.8]{minc}. The relevant boundary curve (actually a straight
line segment from $1$ to $\rme^{2\pi i/n}$) is the case $q=0$, $p=1$
and $r=n$ of $\lam^q(\lam^p-t)^r=(1-t)^r$, where $0\leq t\leq 1$.
The small part of the theorem of Karpelevi\v{c} that we need was
proved by Dmitriev and Dynkin; see \cite[Chap.~8,
Theorem~1.7]{minc}.
\end{proof}

We turn now to the question of uniqueness. The first example of a
Markov matrix $A$ that can be written in the form $A=\rme^B$ for two
different Markov generators was given by Speakman in \cite{speak};
Example~\ref{twogen} provides another. The initial hypothesis of our
next result holds for most embeddable Markov matrices by
Theorem~\ref{embeddistinct}.

\begin{corollary}\label{lemmalog}
Let $A$ be an invertible $n\times n$ Markov matrix with distinct
eigenvalues $\lam_1,\ldots, \lam_n$.
\begin{enumerate}
\item The solutions of $\rme^B=A$ form a discrete set and they all commute with each other and with $A$.
\item Only a finite number of the solutions of $\rme^B=A$ can be Markov generators.
\item If
\begin{eq}
|\lam_r|> \exp(-\pi\tan(\pi/n))\label{lambdamin}
\end{eq}
for all $r$ then only one of the solutions of $\rme^B=A$ can be a
Markov generator, namely the principal logarithm.
\end{enumerate}
\end{corollary}

\begin{proof}
\begin{enumerate}
\item
Since $\Spec(A)=\exp(\Spec(B))$, each $B$ must have $n$ distinct
eigenvalues $\mu_1,\ldots,\mu_n$ and the corresponding eigenvectors
form a basis in $\C^n$. These eigenvectors are also eigenvectors for
$A$ and
\begin{equation}
\lam_r=\rme^{\mu_r}\label{takelog}
\end{equation}
for all $r$. It follows that $B$ can be written as a polynomial
function of $A$. For each $\lam_r$, the equation (\ref{takelog}) has
a discrete set of solutions $\mu_r$.
\item
If $A$ is an invertible Markov matrix with distinct eigenvalues and
the solution $B$ of $\rme^B=A$ is a Markov generator then every
eigenvalue $\mu_r$ of $B$ lies in the sector $\{ -u+iv: u\geq 0,
|v|\leq u\cot(\pi/n)\}$ by (\ref{karpbound}). Combining this
restriction on the imaginary parts of the eigenvalues with
(\ref{takelog}) reduces the set of such $B$ to a finite number. See
\cite[Theorem~6.1]{israel} for items~1 and 2, and for an algorithm
implementing item~2.
\item
We continue with the assumptions and notation of item~2. The
assumption (\ref{lambdamin}) implies that if $\mu_r=-u_r+iv_r$ then
$u_r <\pi\tan(\pi/n)$. Item~2 now yields $|v_r|<\pi$. Hence $\mu_r$
is the principal logarithm of $\lam_r$ and $B$ is the principal
logarithm of $A$.
\end{enumerate}
\end{proof}

The conclusions of the above corollary do not hold if $A$ has
repeated eigenvalues or a non-trivial Jordan form; see
\cite{cuth,israel}. For example the $n \times n$ identity matrix has
a continuum of distinct logarithms $B$ which do not all commute; if
the eigenvalues of $B$ are chosen to be $\{ 2\pi r i:1\leq r\leq
n\}$, then the possible $B$ are parametrized by the choice of an
arbitrary basis as its set of eigenvectors. The general
classification of logarithms is given in \cite{gant} and
\cite[Theorem~1.28]{highambook}. These comments reveal a numerical
instability in the logarithm of a matrix if it has two or more
eigenvalues that are very close to each other.

The following provides a few other conditions that imply the
uniqueness of a Markov generator $B$ such that $A=\rme^B$.

\begin{theorem}[Cuthbert, \cite{cuth1,cuth}]
Let $A=\rme^B$ where $B$ a Markov generator. Then
(\ref{one})$\,\Rightarrow$(\ref{two})$\,\Rightarrow%
$(\ref{three})$\,\Rightarrow$(\ref{four}), where
\begin{eqnarray}
\rme^{-\pi}&<& \det(A)\,\,\leq\,\, 1,\label{one}\\
-\pi &<& \tr(B)\,\,\leq \,\,0,\label{two}\\
\norm B+\bet I\norm&<& \pi,\hspace{2em}\mbox{ where }\bet=\max_{1\leq i\leq n}\{ |B_{i,i}|\}, \label{three}\\
\Spec(B)&\subseteq & \{ z:\Im(z)|<\pi\}. \label{four}
\end{eqnarray}
If $A$ is a Markov matrix that has distinct eigenvalues and
$\det(A)>\rme^{-\pi}$ then its only possible Markov generator is its
principal logarithm $\log(A)$.
\end{theorem}

\begin{proof}\\
(\ref{one})$\Rightarrow$(\ref{two})
This uses $\det(A)=\rme^{\tr(B)}$.\\
(\ref{two})$\Rightarrow$(\ref{three})
This uses the fact that $B+\bet I$ has non-negative entries and its row sums all equal $\bet$, which satisfies $\bet <\pi$.\\
(\ref{three})$\Rightarrow$(\ref{four}) follows directly from
$\Spec(B+\bet I)\subseteq \{z:|z|<\pi\}$.\\
The final statement of the theorem follows the proof of
Corollary~\ref{lemmalog}.
\end{proof}

If $A^t$ is a one-parameter Markov semigroup then for every $t>0$
one may define $L(t)$ to be the number of Markov generators $B$ such
that $\rme^{Bt}=A^t$. Some general theorems concerning the
dependence of $L(t)$ on $t$ may be found in \cite{cuth,singer}.

\section{Regularization}

Let $\cG$ denote the set of  $n\times n$ Markov generators;
following the notation of Lemma~\ref{Markovlog}, $\cG$ is the set of
$G\in \cL$ such that $G_{i,j}\geq 0$ whenever $i\not= j$.

Let $A$ be a Markov matrix satisfying the assumptions of
Lemma~\ref{Markovlog}, for which $L=\log(A)$ does not lie in $\cG$.
There are several \emph{regularizations} of $L$, that is algorithms
that replace $L$ by some $G\in \cG$ that are (nearly) as close to
$L$ as possible. Kreinin and Sidelnikova \cite{KS} have compared
different regularization algorithms for several empirical examples
arising in finance and it appears that they all have similar
accuracy. The best approximation must surely depend on the matrix
norm used, but if one considers the physically relevant matrix norm (\ref{matrixnorm}) then we prove that the simplest method, called diagonal adjustment in \cite{KS}, also produces the best possible approximation. We emphasize that although $\cG$ is a closed convex cone, this does not imply that the best approximation is unique, because the matrix norm (\ref{matrixnorm}) is not strictly convex.

\begin{theorem}\label{approxtheorem}
Let $L\in \cL$ and define $B\in \cG$ by
\[
B_{i,j}=\left\{ \begin{array}{ll}
L_{i,j}&\mbox{ if }i\not= j\mbox{ and } L_{i,j}\geq 0,\\
0&\mbox{ if }i\not= j\mbox{ and } L_{i,j}< 0,
\end{array}\right.
\]
together with the constraint $\sum_{j=1}^nB_{i,j}=0$ for all $i$.
Then
\[
\norm L-B\norm =\min\{ \norm L-G\norm: G\in\cG\}.
\]
\end{theorem}

\begin{proof}
It follows from the definition of the matrix norm that we can deal
with the matrix rows one at a time. We therefore fix $i$ and put
\begin{eqnarray*}
\ell_j&=& L_{i,j},\\
P&=& \{j:j\not= i\mbox{ and } \ell_j\geq 0\},\\
N&=& \{j:j\not= i\mbox{ and } \ell_j< 0\},\\
\ell_P&=& \sum_{j\in P} \ell_j\geq 0,\\
\ell_N&=& -\sum_{j\in N} \ell_j\geq 0,\\
\end{eqnarray*}
so that $\ell_i=\ell_N-\ell_P$. We next put $b_j=B_{i,j}$, where $B$
is defined as in the statement of the theorem.  Thus
\[
b_j=\left\{ \begin{array}{ll}
\ell_j &\mbox{ if } j\in P,\\
0&\mbox{ if } j\in N,\\
-\ell_P&\mbox{ if } j=i.
\end{array}\right.
\]

A direct calculation shows that
\[
\ell_j-b_j=\left\{ \begin{array}{ll}
0 &\mbox{ if } j\in P,\\
\ell_j&\mbox{ if } j\in N,\\
\ell_N&\mbox{ if } j=i.
\end{array}\right.
\]
Therefore
\[
\norm \ell- b\norm_1=2\ell_N.
\]
Finally given $G\in \cG$ we define $g_j=G_{i,j}$ for all $j$. We
have
\begin{eqnarray*}
\norm \ell- g\norm_1&=& \sum_{j\not= i}|\ell_j-g_j|+|\sum_{j\not=
i}(\ell_j- g_j)|\\
&\geq & \sum_{j\in P}|\ell_j-g_j|+\sum_{j\in N}|\ell_j-g_j|\\
&& +\,\rule[-2ex]{0.05ex}{5ex} \sum_{j\in
N}(\ell_j-g_j)\,\rule[-2ex]{0.05ex}{5ex}-\rule[-2ex]{0.05ex}{5ex}
\sum_{j\in
P}(\ell_j-g_j)\,\rule[-2ex]{0.05ex}{5ex}\\
&\geq & \sum_{j\in P}|\ell_j-g_j|+\ell_N\\
&& +\,\rule[-2ex]{0.05ex}{5ex} \sum_{j\in
N}\ell_j\,\rule[-2ex]{0.05ex}{5ex}-
\sum_{j\in P}|\ell_j-g_j|\\
&=& 2\ell_N.
\end{eqnarray*}
\end{proof}

\begin{theorem}\label{approxtheorem}
Let $A$ be a Markov matrix such that $\Spec(A)\cap
(-\infty,0]=\emptyset$ and put $L=\log(A)$. If $B\in \cG$ and $\norm
L-B\norm =\eps$ then
\[
\norm A-\rme^B\norm \leq \min\{ 2,\rme^\eps -1\}\leq \min\{ 2,2\eps
\}.
\]
\end{theorem}

\begin{proof}
If we put $E=L-B$ then the series expansion
\[
\rme^L=\rme^B+\int_{t=0}^1 \rme^{B(1-t)}E\rme^{Bt}\,\rmd t%
+\int_{t=0}^1\int_{s=0}^t
\rme^{B(1-t)}E\rme^{B(t-s)}E\rme^{Bs}\,\rmd s\rmd t+\ldots
\]
given in \cite[Theorem~11.4.1]{LOTS} yields
\begin{eqnarray*}
\norm A-\rme^B\norm &=&\norm \rme^L -\rme^B\norm \\
&\leq& \norm E\norm +\norm E\norm^2/2! +\ldots\\
&=& \rme^{\norm E \norm} -1.
\end{eqnarray*}
The other part of the estimate uses $\norm A\norm =1$ and
$\norm\rme^B\norm =1$.
\end{proof}

\section{Some Numerical Examples}

\begin{example}\end{example}
The Markov matrix
\[
A=\left[ \begin{array}{ccc}
0.30&0.45&0.25\\
0.14&0.84&0.02\\
0.14&0.52&0.34
\end{array}\right].
\]
has eigenvalues $1,0.32, 0.16$, exactly. The matrix $L=\log(A)$ is
given to four decimal places by
\[
L=\left[ \begin{array}{ccc}
-1.5272&0.5991&0.9281\\
0.3054&-0.2371&-0.0683\\
0.3054&0.9023&-1.2078
\end{array}\right],
\]
and has a negative off-diagonal entry. The closest Markov generator
$B$ to $L$ as described above is
\[
B=\left[ \begin{array}{ccc}
-1.5272&0.5991&0.9281\\
0.3054&-0.3054&0\\
0.3054&0.9023&-1.2078
\end{array}\right].
\]
and the embeddable Markov matrix $\widetilde{A}=\rme^B$ (where $B$
is entered to full precision) is given by
\[
\widetilde{A}=\left[ \begin{array}{ccc}
0.3000&0.4383&0.2617\\
0.1400&0.8046&0.0554\\
0.1400&0.5057&0.3543
\end{array}\right].
\]
One observes that all the entries of $A-\widetilde{A}$ are less than
$0.036$ in absolute value.\block

The following exactly soluble example illustrates the use of some of
our theorems.

\begin{theorem}
Let
\[
L_s=\left[ \begin{array}{ccc}
-1-s&1&s\\
s&-1-s&1\\
1&s&-1-s
\end{array}\right]
\]
where $s\in\R$, and let $A_s=\rme^{L_s}$. Then
\begin{enumerate}
\item If $s\geq 0$ then $A_s$ is an embeddable Markov matrix.

\item If $s< \sig\sim -0.5712$ then $A_s$ has at least one negative entry.

\item If $\sig\leq s<0$ then $A_s$ is Markov but not embeddable.
\end{enumerate}
\end{theorem}

\begin{proof}
We first note that $L_s1=0$ for every $s\in\R$, so $A_s1=1$.

Item~1 follows from the fact that $L_s$ satisfies all the conditions
for the generator of a Markov semigroup.

To prove item~2 we note that $L_s=-(1+s)I+F+sB$ where $F,\, B$ are
permutation matrices that commute. Let $\cS$ be the set of all $s$
such that $\rme^{L_s}$ is non-negative. If $t\in \cS$ and $s\geq t$
then
\[
L_s=L_t+ (s-t)B-(s-t)I
\]
where all the matrices commute, and
\[
\rme^{L_s}=\rme^{-(s-t)}\rme^{L_t}\rme^{B(s-t)}\geq 0,
\]
so $s\in\cS$. Therefore $\cS$ is an interval, which is obviously
closed. Numerical calculations show that the smallest number
$\sig\in\cS$ is approximately $-0.5712$. More rigorously if $s<-1$
then
\[
\det(A_s)=\rme^{\tr(L_s)}=\rme^{-3(1+s)}>1
\]
so $A_s$ cannot be a Markov matrix and must have a negative entry.
This establishes that $\sig\geq -1$.

Clearly $L_s$ is not a Markov generator if $\sig\leq s<0$. We prove
item~3 by obtaining a contradiction from the existence of a Markov
generator $B_s$ with $\rme^{B_s}=A_s$. Since
\[
\exp(\Spec(L_s))=\Spec(A_s) = \exp(\Spec(B_s))
\]
we conclude that every eigenvalue of $L_s$ differs from an
eigenvalue of $B_s$ by an integral multiple of $2\pi i$. A direct
computation shows that
\[
\Spec(L_s)=\left\{0,-\frac{3(1+s)}{2} \pm \frac{\sqrt{3}
(1-s)i}{2}\right \}.
\]
For $s$ in the stated range, each non-zero eigenvalue $\lam$ of
$L_s$ satisfies $|\arg(\lam)|<5\pi/6$ and the same applies if one
adds an integral multiple of $2\pi i$ to the eigenvalue. Hence each
non-zero eigenvalue $\lam$ of $B_s$ satisfies $|\arg(\lam)|<5\pi/6$
and (\ref{argbound}) implies that $B_s$ cannot be a generator.
\end{proof}

\begin{example} \end{example}
The following illustrates the difficulties in dealing with Markov
matrices that have negative eigenvalues. If $c=2\pi/\sqrt{3}$ and
\[
B=c\left[ \begin{array}{ccc}
-1&1&0\\
0&-1&1\\
1&0&-1
\end{array}\right],
\]
then the eigenvalues of $B$ are $0,\, -\sqrt{3}\pi\pm \pi i$. The
matrix $A=\rme^B$ is self-adjoint with eigenvalues
$1,-\rme^{-\sqrt{3}\pi},-\rme^{-\sqrt{3}\pi}$. If one uses Matlab's
`logm' command to compute $\log(A)$, one obtains a matrix with
complex entries that is not close to $B$; it might be considered
that `logm' should produce a real logarithm of a real matrix if one
exists, but it is not easy to see how to achieve this.

\begin{example}\label{twogen} \end{example}

We continue with the above example, but with the more typical choice
$c=4$. The eigenvalues of $B$ are now $0,\, -6\pm 3.4641i$. Clearly
$A=\rme^B$ is an embeddable Markov matrix. If one rounds to four
digits one obtains
\[
A\sim\left[ \begin{array}{ccc}
 0.3318  &  0.3337  &   0.3346\\
    0.3346   &  0.3318  &   0.3337\\
    0.3337   &  0.3346   &  0.3318
\end{array}\right],
\]
The use of `logm' yields
\[
\log(A)\sim\left[ \begin{array}{ccc}
 -4  &  0.3724  &  3.6276\\
    3.6276  & -4  &  0.3724\\
    0.3724 &   3.6276 &  -4
\end{array}\right],
\]
which is also a Markov generator. We conclude that $A$ is an
embeddable Markov matrix in (at least) two distinct ways.

This is not an instance of a general phenomenon. If one defines the
$5\times 5$ cyclic matrix $B$ by
\[
B_{r,s}=\left\{ \begin{array}{cc} -4&\mbox{ if } r=s,\\
4&\mbox{ if } r+1=s,\\
4&\mbox{ if } r=5,\, s=1,\\
0&\mbox{ otherwise, }
\end{array}\right.
\]
then $B$ is a Markov generator with eigenvalues $0,\, -7.2361\pm
2.3511i,\, -2.7639\pm3.8042i$. However $L=\log(\exp(B))$, with the
principal choice of the matrix logarithm, is a cyclic matrix with
some negative off-diagonal entries, so it cannot be a Markov
generator. \block

{\bf Acknowledgements} We would like to thank N J Higham and L Lin for
many valuable comments and references to the literature.

Department of Mathematics\\
King's College\\
Strand\\
London\\
WC2R 2LS\\
UK

E.Brian.Davies@kcl.ac.uk


\begin{thebibliography}{99}

\bibitem{chung}
K L Chung, Markov chains with stationary transition probabilities,
Springer, Berlin, 1960.

\bibitem{cuth1}
J R Cuthbert, On uniqueness of the logarithm for Markov semigroups,
J. London Math. Soc. 4 (1972) 623-630.

\bibitem{cuth}
J R Cuthbert, The logarithm function for finite-state Markov
semi-groups, J. London Math. Soc. (2), 6 (1973) 524-532.

\bibitem{EBDtriv}
E B Davies, Triviality of the peripheral point spectrum, J. Evol.
Eqns. 5 (2005) 407-415.

\bibitem{LOTS}
E B Davies, Linear Operators and Their Spectra, Cambridge Univ.
Press, 2007.

\bibitem{DS}
N Dunford and J T Schwartz, Linear Operators, Part 1, Interscience
Publ., New York, 1966.

\bibitem{elfving}
G Elfving, Zur Theorie der Markoffschen Ketten, Acta Soc. Sci.
Fennicae, n. Ser. A2 no. 8, (1937) 1-17.

\bibitem{gant}
F R Gantmacher, The Theory of Matrices, vol.~1, Chelsea, New York,
1959.

\bibitem{highambook}
N J Higham, Functions of Matrices: Theory and Computation, SIAM,
Philadelphia, PA, USA, 2008.

\bibitem{higham}
N J Higham and L Lin, On $p$th roots of stochastic matrices. MIMS
preprint, 2009.

\bibitem{israel}
R B Israel, J S Rosenthal, J Z Wei, Finding generators for Markov
chains via empirical transition matrices, with applications to
credit ratings, Math. Finance 11, no. 2 (2001) 245-265.

\bibitem{johan}
S Johansen, A central limit theorem for finite semi-groups and its
application to the imbedding problem for finite-state Markov chains,
Zeit. Wahrsch. 26 (1973) 171-90.

\bibitem{karp}
F I Karpelevi\v{c}, On the characteristic roots of matrices with
nonnegative elements, Izv. Akad. Nauk SSSR Ser. Mat. 15 (1951)
361-383 (in Russian), Amer. Math. Soc. Transl. (2) 140 (1988)
79-100.

\bibitem{kingman}
J F C Kingman, The imbedding problem for finite Markov chains, Zeit.
Wahrsch. 1 (1962) 14-24.

\bibitem{KS}
A Kreinin and M Sidelnikova, Regularization algorithms for
transition matrices, Algo Research Quarterly, 4 (2001) 23-40.

\bibitem{minc}
H Minc, Nonnegative Matrices, Wiley, New York, 1988.

\bibitem{nagel}
R Nagel, One-Parameter Semigroups for Positive Operators, Lect.
Notes Math. 1184, Springer, New York, 1986.

\bibitem{runn}
J Th Runnenberg 1962. On Elfving's problem of imbedding a
time-discrete Markov chain in a continuous time one for finitely
many states, Proc. Koninklijke Nederlandse Akademie van
Wetenschappen, Ser. A, Math. Sci. 65 (5) (1962) 536-41.

\bibitem{singer}
B Singer, S Spilerman, The representation of social processes by
Markov models, Amer. J. Sociology, 82 (1976) 1-54.

\bibitem{speak}
J M O Speakman, Two Markov chains with a common skeleton, Zeit.
Wahrsch. 7 (1967) 224.

\bibitem{zahl}
S Zahl, A Markov process model for follow-up studies, Human Biology,
27 (1955) 90-120.

\end{thebibliography}
\end{document}